\documentclass{amsart}

\usepackage{amssymb}
\usepackage{amsthm}
\usepackage{amsfonts}
\usepackage{mathrsfs}
\usepackage{amsmath}

\newtheorem{theorem}{Theorem}[section]
\newtheorem{corollary}[theorem]{Corollary}
\newtheorem{lemma}[theorem]{Lemma}
\newtheorem{proposition}[theorem]{Proposition}
\newtheorem{definition}{Definition}
\newtheorem*{remark}{Remark}
\numberwithin{equation}{section}

\newcommand{\RePt}{\mathrm{Re}\,}
\newcommand{\ImPt}{\mathrm{Im}\,}
\newcommand{\ball}{\mathbb{B}}
\newcommand{\sphere}{\mathbb{S}}

\newcommand{\bfL}{\mathcal{L}}

\newcommand{\calU}{\mathcal{U}}

\newcommand{\bfi}{\mathbf{i}}
\newcommand{\bfbeta}{\boldsymbol{\beta}}
\newcommand{\bfrho}{\boldsymbol{\rho}}
\newcommand{\calB}{\mathcal{B}}

\newcommand{\bbC}{\mathbb{C}}
\newcommand{\bbR}{\mathbb{R}}

\newcommand{\bbN}{\mathbb{N}}

\begin{document}

\title[A uniqueness property for Bergman functions]
{A uniqueness property for Bergman functions \\
on the Siegel upper half-space}

\author{Congwen Liu}
\email{cwliu@ustc.edu.cn}
\address{CAS Wu Wen-Tsun Key Laboratory of Mathematics,
School of Mathematical Sciences,
University of Science and Technology of China\\
Hefei, Anhui 230026,
People's Republic of China.}
\thanks{
The first author was supported by the National Natural Science Foundation of China grant 11971453.}

\author{Jiajia Si}
\email{sijiajia@mail.ustc.edu.cn}
\address{School of Science, Hainan University,
Haikou, Hainan 570228,
People's Republic of China.}
\thanks{The second author was supported by the Hainan Provincial Natural Science Foundation of China grant 120QN177.}

\author{Heng Xu}
\email{xuheng86@mail.ustc.edu.cn}
\address{School of Mathematical Sciences,
University of Science and Technology of China\\
Hefei, Anhui 230026,
People's Republic of China.}


\begin{abstract}
In this paper, we show that the Bergman functions on the Siegel upper half-space enjoy
the following uniqueness property:  if $f\in A_t^p(\calU)$ and $\bfL^{\alpha} f\equiv 0$
for some nonnegative multi-index $\alpha$, then $f\equiv 0$, where $\bfL^{\alpha}:=(\bfL_1)^{\alpha_1}
\cdots (\bfL_n)^{\alpha_n}$ with $\bfL_j =  \frac{\partial }{\partial z_j}  + 2i \bar{z}_j \frac{\partial }{\partial z_n}$ for $j=1,\ldots, n-1$ and $\bfL_n = \frac{\partial }{\partial z_n}$. As a consequence, we obtain a new
integral representation for the Bergman functions on the Siegel upper half-space.
In the end, as an application, we derive a result that relates the Bergman norm to a ``derivative norm'',
which suggests an alternative definition of the Bloch space and a notion of the Besov spaces
over the Siegel upper half-space.
\end{abstract}


\maketitle

\section{introduction}

Let $\mathbb{C}^n$ be the $n$-dimensional complex Euclidean space.
The Siegel upper half-space is the set
\[
\calU=\left\{ z\in\mathbb{C}^n:\ImPt z_n>|z^{\prime}|^2\right\}.
\]
Here and throughout the paper, we use the notation
\[
z=(z^{\prime},z_n),\,\,\,\, \text{where}\, z^{\prime}=(z_1,\cdots,z_{n-1})\in\mathbb{C}^{n-1}\,\, \text{and}\,\, z_n\in\mathbb{C}^{1}.
\]

For $1\leq p<\infty$ and $t>-1$, the space $L_t^p(\calU)$ consists of all Lebesgue measurable
functions $f$ on $\calU$ for which
\begin{equation*}
\|f\|_{p,t}=\bigg\{\int\limits_{\calU} |f(z)|^p \bfrho(z)^t dV(z)\bigg\}^{1/p}<\infty,
\end{equation*}
where $dV$ denotes the Lebesgue measure on $\mathbb{C}^n$.
The weighted Bergman space $A_t^p(\calU)$ is the closed subspace of $L_t^p(\calU)$ consisting of
holomorphic functions.
As usual, while $t=0$ it is written as $A^p(\calU)$.
In particular, the Bergman space $A^2(\calU)$ is a Hilbert space
endowed with the usual $L^2$ inner product. There is a self-adjoint projection of $L^2(\calU)$ onto $A^2(\calU)$, which can be expressed as an integral operator:
\[
Pf(z)=\frac{n!}{4\pi^n}\int\limits_{\calU} \frac{f(w)}{\bfrho(z,w)^{n+1}}dV(w),
\]
where
\[
\bfrho(z,w)= \frac{i}{2}(\overline{w}_n-z_n)-z^{\prime} \cdot \overline{w^{\prime}} .
\]
See \cite[Theorem 5.1]{Gin64} for instance.


Note that $\bfrho(z):=\bfrho(z,z)$ is just the defining function for $\calU$.
For $n \geq 2$, a simple calculation shows
\begin{align*}
\bfL_j :=& \frac {\partial }{\partial z_j}  + 2i \bar{z}_j \frac {\partial }{\partial z_n}, \quad j=1,\ldots,n-1
\end{align*}
forms a global basis for the space of tangential $(1, 0)$ vector fields on the boundary $b\calU$.
Also, we write $\bfL_n := \frac {\partial }{\partial z_n}$.
As usual, for $\alpha\in \mathbb{N}_0^n$ we write $\bfL^{\alpha}:=(\bfL_1)^{\alpha_1} \cdots (\bfL_n)^{\alpha_n}$.
Here and throughout the paper, $\mathbb{N}_0$ denotes the set of nonnegative integers.

Our first main result is the following, we refer to it as the uniqueness property for Bergman functions on $\calU$.

\begin{theorem}\label{thm:main2}
Suppose that $1\leq p<\infty$, $t>-1$ and $f\in A_t^p(\calU)$. If $\bfL^{\alpha} f\equiv 0$ for some multi-index $\alpha\in \mathbb{N}_0^n$, then $f\equiv 0$.
\end{theorem}

Theorem \ref{thm:main2} shows how differently Bergman functions on $\calU$ behave when compared with Bergman functions on bounded domains.
This is inspired by \cite[Theorem 2.6]{RY96}, which states that, if $u$ is a harmonic Bergman function
on the upper half-space of $\mathbb{R}^n$ and $\alpha$ is a multi-index,
then $D^{\alpha} u\equiv 0$ implies $u\equiv 0$. However, the situation here is quite different. Note that in Theorem \ref{thm:main2},
the differential operators $\bfL_j$ cannot be replaced by $\frac {\partial }{\partial z_j}$, $j=1,\ldots, n-1$.
For example, $f(z)=(z_n+i)^{-(n+2)}$ lies in $A^p(\calU)$
for any $1\leq p<\infty$, but clearly $\frac {\partial f}{\partial z_j}  \equiv 0$ for $j=1,2,\ldots,n-1$.

%

An analogous result holds for the Hardy functions.
Recall that the Hardy space $H^p(\calU)$ consists of all functions $f$ holomorphic on $\calU$ such that
\begin{equation*}
\|f\|_{H^p(\calU)}=\sup_{\epsilon >0} \Bigg\{ \int\limits_{b\calU} f(u+\epsilon\, \bfi)^p d\bfbeta(u) \Bigg\}^{1/p}<\infty,
\end{equation*}
where $b\calU$ denotes the boundary of $\calU$ and the measure $d\bfbeta$ on $b\calU$ is defined by
the formula
\begin{equation*}
\int\limits_{b\calU} f d\bfbeta = \int\limits_{\bbC^{n-1}\times \bbR}
f(z^{\prime}, t+i|z^{\prime}|^2) dz^{\prime} dt.
\end{equation*}

Since $H^p(\calU) \subset A^{\frac{(n+1)p}{n}}(\calU)$ for all $1\leq p < \infty$ (see the appendix), we have the following.

\begin{corollary}\label{thm:main3}
Suppose that $1\leq p<\infty$ and $f\in H^p(\calU)$. If $\bfL^{\alpha} f\equiv 0$ for some multi-index $\alpha\in \mathbb{N}_0^n$,
then $f\equiv 0$.
\end{corollary}

Theorem \ref{thm:main2} enables us to obtain a new integral representation of Bergman functions on $\calU$. Let $\mathcal{P}_{\lambda}$ be the integral operator given by
\[
\mathcal{P}_{\lambda}f(z) ~=~ c_{\lambda} \int\limits_{\calU} \frac {\bfrho(w)^{\lambda}}{\bfrho(z,w)^{n+1+\lambda}} f(w) dV(w),
\]
where
\[
c_{\lambda} =\frac{\Gamma(n+1+\lambda)}{4\pi^n \Gamma(1+\lambda)}.
\]

\begin{theorem}\label{thm:main}
Suppose that $1\leq p < \infty$, $t>-1$ and $\lambda\in \mathbb{R}$ satisfy
\[
\begin{cases}
\lambda > \frac {t+1}{p}-1,& 1<p<\infty,\\
\lambda \geq t,& p=1.
\end{cases}
\]
 If $f\in A_t^p(\calU)$ then
\begin{equation*}\label{eqn:intrepn1}
f=  \frac{(2i)^N\Gamma(1+\lambda)}{\Gamma(1+\lambda+N)}\mathcal{P}_{\lambda} (\bfrho^N \bfL_n^N f)
\end{equation*}
for any $N\in\bbN_0$.
\end{theorem}

When $N=0$, this reduces to the known integral representation for Bergman functions on $\calU$ by 
Djrbashian and Karapetyan (\cite[Theorem 2.1]{DK93}).

As an application of Theorem \ref{thm:main}, we give a useful result that relates the Bergman
norm to a ``derivative norm''.
Throughout the paper we will abbreviate inessential constants involved in inequalities by writing $A\lesssim B$ for positive quantities $A$ and $B$ if the ratio $A/B$ has a positive upper bound. Also, $A\approx B$ means both $A\lesssim B$ and $B\lesssim A$.

\begin{theorem}\label{thm:main4}
Let $1\leq p<\infty$, $t>-1$, $N\in \bbN_0$ and $\alpha\in\bbN_0^n$. Then
\[
\|f\|_{p,t} \approx \|\bfrho^N \bfL_n^N f\|_{p,t} \approx \sum_{|\alpha | = N} \| \bfrho^{\langle \alpha \rangle} \bfL^{\alpha} f\|_{p,t}
\]
as $f$ ranges over $A_t^p(\calU)$. Here $\langle \alpha \rangle=|\alpha^{\prime}|/2+\alpha_n$.
\end{theorem}

Note the proviso ``as $f$ ranges over $A_t^p(\calU)$'' in this theorem; the norm equivalence
stated here would fail if we allowed $f$ to vary over all holomorphic functions on $\calU$.

Theorem \ref{thm:main4} suggests an alternative notion of the Bloch space $\calB$ on $\calU$, which is defined to be the space of holomorphic functions $f$ such that
\[
\| f\|_{\mathcal{B}} := \sum_{|\alpha | = N} \| \bfrho^{\langle \alpha \rangle} \bfL^{\alpha} f\|_\infty<\infty
\]
for any fixed $N\in \mathbb{N}$.
This notion seems more ``agreeable'' than the one given by B\'ekoll\'e in \cite{Bek83}.
Similarly,  we define the Besov spaces $B_p(\calU)$ on $\calU$ to be the spaces of the functions $f$ with
\[
\|f\|_{B_p}:=\sum_{|\alpha | = N} \| \bfrho^{\langle \alpha \rangle} \bfL^{\alpha} f\|_{p,d\tau}<\infty
\]
for sufficiently large $N\in \mathbb{N}$, where $d\tau=\bfrho^{-n-1}dV$.
This generalizes the notion of the Besov spaces on the half-plane given by Semmes \cite{Sem84} to higher dimensions.
The details will be presented in our forthcoming papers.

The paper is organized as follows. We collect several auxiliary lemmas in Section 2. In Section 3 we study the boundedness of the map $f\mapsto \bfrho^{\langle \alpha \rangle} \bfL^{\alpha} f$, which is vital to the proofs of our main results. Sections 4 and 5 
are devoted to the proofs of Theorems \ref{thm:main2} and \ref{thm:main}. We prove Theorem \ref{thm:main4} in Section 6. 
At the end, the proof of the fact $H^p(\calU) \subset A^{\frac{(n+1)p}{n}}(\calU)$ is included in Appendix.

\section{Preliminaries}

In the section we collect several auxiliary lemmas.
The first lemma is the traditional integral representation of Bergman functions on $\calU$.

\begin{lemma}[{\cite[Theorem 2.1]{DK93}}]\label{thm:DK}
Suppose that $1\leq p < \infty$, $t>-1$ and $\lambda\in \mathbb{R}$ satisfy
\[
\begin{cases}
\lambda > \frac {t+1}{p}-1,& 1<p<\infty,\\
\lambda \geq t,& p=1.
\end{cases}
\]
If $f\in A_{t}^p(\calU)$ then $f=\mathcal{P}_{\lambda} f$.
\end{lemma}

\begin{lemma}[{\cite[Theorem 3.1]{DK93}}]\label{thm3.1:DK}
Suppose that $1\leq p < \infty$, $t>-1$ and $\lambda\in \mathbb{R}$. If $p(\lambda+1)>t+1$, then $\mathcal{P}_{\lambda}$ is a bounded projection from $L_t^p(\calU)$ to $A_t^p(\calU)$.
\end{lemma}

In fact, the condition $p(\lambda+1)>t+1$ is also necessary for the boundedness of $\mathcal{P}_{\lambda}$. See \cite[Theorem 1]{LLHZ19}.

\begin{lemma}[{\cite[Key Lemma]{LLHZ19}}]\label{lem:keylem2}
Suppose that $r,\,s>0$, $t>-1$ and $r+s-t>n+1$. Then
\begin{equation}\label{eqn:keylem2}
\int\limits_{\calU}  \frac {\bfrho(w)^{t}} {\bfrho(z,w)^{r} \bfrho(w,u)^{s}} dV(w)
~=~ \frac {C_1(n,r,s,t)} {\bfrho(z,u)^{r+s-t-n-1}}
\end{equation}
holds for all $z, u\in \calU$, where
\begin{equation*}\label{eqn:const}
C_1(n,r,s,t) :=  \frac {4\pi^{n} \Gamma(1+t)\Gamma(r+s-t-n-1)}{\Gamma(r)
\Gamma(s)}.
\end{equation*}
\end{lemma}

\begin{lemma}[{\cite[Lemma 5]{LLHZ19}}]
Let $s,t\in \mathbb{R}$. Then we have
\begin{equation}\label{eqn:keylem}
\int\limits_{\calU} \frac{\bfrho(w)^{t}} {|\bfrho(z,w)|^{s}} dV(w) ~=~
\begin{cases}
\dfrac {C_2(n,s,t)} {\bfrho(z)^{s-t-n-1}}, &
\text{ if } t>-1 \text{ and } s-t>n+1\\[12pt]
+\infty, &  otherwise
\end{cases}
\end{equation}
for all $z\in \calU$, where
\[
C_2(n,s,t):=\frac {4 \pi^{n} \Gamma(1+t) \Gamma(s-t-n-1)} {\Gamma^2\left(s/2\right)}.
\]
\end{lemma}

\begin{lemma}[{\cite[Lemma 2.1]{LS20}}]
We have
\begin{equation}\label{eqn:elemtryeq1}
 2|\bfrho(z,w)|\geq \max\{\bfrho(z),\bfrho(w)\}
\end{equation}
for any $z, w\in \calU$.
\end{lemma}

\section{The boundedness of the map $f\mapsto \bfrho^{\langle \alpha \rangle} \bfL^{\alpha} f$}

\begin{proposition}\label{prop:bddnsofder}
Suppose $1\leq p < \infty$ and $t>-1$. Then for any $\alpha\in \bbN_{0}^n$, the map $f\mapsto \bfrho^{\langle \alpha \rangle} \bfL^{\alpha} f$
is a bounded linear operator from $A_t^p(\calU)$ into $L_t^p(\calU)$.
In particular, for any $N\in \mathbb{N}_0$,
the map $f\mapsto \bfL_n^N f$ is a bounded linear operator from $A_t^p(\calU)$ into $A_{t+Np}^p(\calU)$.
\end{proposition}

To prove Proposition \ref{prop:bddnsofder}, we need two lemmas.
\begin{lemma}\label{lem:ptwsest}
For any fixed $\alpha^{\prime}\in \mathbb{N}_0^{n-1}$, we have
\[
\big|(z^{\prime}-w^{\prime})^{\alpha^{\prime}}\big| \lesssim |\bfrho(z,w)|^{\frac {|\alpha^{\prime}|}{2}}
\]
for all $z,w\in \calU$.
\end{lemma}

\begin{proof}
To each $z\in\calU$, we define the following (holomorphic) affine self-mapping of $\calU$:
\[
h_z(u) ~:=~ \left(u^{\prime} - z^{\prime}, u_n - \RePt z_n - 2i u^{\prime} \cdot \overline{z^{\prime}}  + i|z^{\prime}|^2 \right),
\quad u\in \calU.
\]
All these mappings are holomorphic automorphisms of $\calU$. See \cite[Chapter XII]{Ste93}.
In particular, we have
\[
h_z(z)=\left(0^{\prime}, i \left(\ImPt z_n - |z^{\prime}|^2 \right)\right) = \bfrho(z) \bfi,
\]
where $\bfi=(0^{\prime},i)$.
Also, an easy calculation shows that
\[
\bfrho(h_z(u),h_z(v))=\bfrho(u,v).
\]

Note that
\begin{equation*}\label{eqn:eleinq}
2|\bfrho(u, s\bfi)|=|u_n+is|> \ImPt u_n \geq |u^{\prime}|^2
\end{equation*}
for any $u\in \calU$ and any $s>0$.
Taking $u=h_z(w)$ and $s=\bfrho(z)$ in the above inequality, and associating the previous argument, we get
\[
\big|(h_z(w))^{\prime}\big|^2 \lesssim \left|\bfrho(h_z(w), \bfrho(z) \bfi)\right| = \left|\bfrho(h_z(w), h_z(z))\right|= |\bfrho(w,z)|.
\]
Consequently,
\[
\big|(z^{\prime}-w^{\prime})^{\alpha^{\prime}}\big| \lesssim \big|(z^{\prime}-w^{\prime})\big|^{|\alpha^{\prime}|}
= \big|(h_z(w))^{\prime}\big|^{|\alpha^{\prime}|} \lesssim |\bfrho(z,w)|^{\frac {|\alpha^{\prime}|}{2}},
\]
as desired.
\end{proof}

Given $a,b \in \bbR$ and $\alpha^{\prime}\in \bbN_{0}^{n-1}$ we define the integral operator $T_{a,b,\alpha^{\prime}}$
by
\[
T_{a,b,\alpha^{\prime}} f(z) := \bfrho(z)^{a} \int\limits_{\calU} \frac {
\big|(z^{\prime}-w^{\prime})^{\alpha^{\prime}}\big|  \bfrho (w)^{b} } {|\bfrho(z,w)|^{n+1+a+b+\frac {|\alpha^{\prime}|}{2}}} f(w) dV(w),
\quad z\in \calU.
\]

\begin{lemma}\label{lem:IntOper}
Suppose that $1\leq p\leq \infty$ and $s\in\mathbb{R}$. Then $T_{a,b,\alpha^{\prime}}$ is bounded on $L_s^p(\calU)$ if $-pa<s+1<p(b+1)$.
\end{lemma}

\begin{proof}
In the special case when $\alpha^{\prime}=0^{\prime}$,  this has been shown in \cite[Theorem 1]{LLHZ19}.
To prove the general case,  note that
\[
|T_{a,b,\alpha^{\prime}} f(z)| \lesssim T_{a,b,0^{\prime}} (|f|)(z)
\]
by Lemma \ref{lem:ptwsest}.
\end{proof}

\begin{proof}[{Proof of Proposition \ref{prop:bddnsofder}}]
It suffices to prove the first assertion.
Let $\lambda$ be sufficiently large such that $p(\lambda+1)>t+1$. Then by Lemma \ref{thm:DK}, we have
\[
f(z) = c_\lambda \int\limits_{\calU} \frac{\bfrho(w)^{\lambda}}{\bfrho(z,w)^{n+1+\lambda}} f(w)dV(w).
\]
For any $\alpha\in \mathbb{N}_{0}^n$, it is not hard to check that
\[
\bfL^{\alpha} \left\{\frac {1}{\bfrho(z,w)^{n+1+\lambda}} \right\}  = C(n,\lambda,\alpha)
\frac {\left(\overline{{z}^{\prime}}-\overline{{w}^{\prime}}\right)^{\alpha^{\prime}}}
{\bfrho(z,w)^{n+1+\lambda+|\alpha|}},
\]
where $C(n,\lambda,\alpha)$ is a constant depending on $n$, $\lambda$ and $\alpha$. Also, note that $|\alpha|=\langle \alpha \rangle+|\alpha^{\prime}|/2$.
Therefore,
\begin{align*}
\big|\bfrho^{\langle \alpha \rangle} \bfL^{\alpha} f(z)\big|
~\lesssim~& \bfrho(z)^{\langle \alpha \rangle}
\int\limits_{\calU} \frac { \big|(z^{\prime}-w^{\prime})^{\alpha^{\prime}}\big| \bfrho(w)^{\lambda}}
{|\bfrho(z,w)|^{n+1+\lambda+|\alpha|}} |f(w)| dV(w)\\
~=~& T_{\langle \alpha \rangle, \lambda, \alpha^{\prime}}(|f|)(z).
\end{align*}
Now, the proposition follows immediately from Lemma \ref{lem:IntOper}.
\end{proof}

\section{Proof of Theorem \ref{thm:main2}}


We divide the proof into two steps.

\textit{Step 1.} We first consider the special case $\alpha=(0^{\prime}, N)$ with $N\geq 1$.
This will be proved if we show that $\bfL_n^N f \equiv 0$ implies $\bfL_n^{N-1} f \equiv 0$ and then by induction.

Suppose $\bfL_n^N f \equiv 0$. Then $f$ has the form
\[
f(z)=f_{N-1}(z^{\prime}) z_n^{N-1}+f_{N-2}(z^{\prime}) z_n^{N-2}+\cdots+ f_0(z^{\prime})
\]
with $f_{N-1}, \cdots, f_0$ holomorphic functions of $z^{\prime}$.
It follows that
\[
\bfL_n^{N-1} f(z) = (N-1)! f_{N-1}(z^{\prime}).
\]
We are then reduced to show that $f_{N-1}(z^{\prime})\equiv 0$. In view of Proposition \ref{prop:bddnsofder}, we know that $\bfrho^{N-1} \bfL_n^{N-1} f\in L_t^p(\calU)$. Thus,
\begin{align*}
\infty &> \int\limits_{\calU}  \bfrho(z)^{(N-1)p+t} \left| f_{N-1}(z^{\prime})\right|^p dV(z)\\
&= \int\limits_{\bbC^{n-1}} \Bigg\{\int\limits_{\ImPt z_n > |z^{\prime}|^2}
\left(\ImPt z_n - |z^{\prime}|^2\right)^{(N-1)p+t} dm_2(z_n)\Bigg\} |f_{N-1}(z^{\prime})|^p dm_{2n-2}(z^{\prime}).
\end{align*}
However, observe that the inner integral is divergent for any fixed $z^{\prime}$. This forces that $f_{N-1}(z^{\prime})\equiv 0$.

\textit{Step 2.} We then consider the general case.
Suppose that $\bfL^{\alpha} f\equiv 0$. A simple calculation shows that
\[
\bfL^{\alpha} f  ~=~ \sum_{0^{\prime}\leq \gamma^{\prime}\leq \alpha^{\prime}} \binom{\alpha^{\prime}}{\gamma^{\prime}}  (2i\bar{z}^{\prime})^{\gamma^{\prime}} \partial^{\alpha-(\gamma^{\prime}, -|\gamma^{\prime}|)} f.
\]
View this as a polynomial in $\bar{z}^{\prime}$. So, the assumption $\bfL^{\alpha} f\equiv 0$ implies that
$\partial^{\alpha-(\gamma^{\prime}, -|\gamma^{\prime}|)} f \equiv 0$
for all $\gamma^{\prime}\in \mathbb{N}_{0}^{n-1}$ with $0^{\prime}\leq \gamma^{\prime}\leq \alpha^{\prime}$.
In particular, when $\gamma^{\prime} = \alpha^{\prime}$, we have $\bfL_n^{|\alpha|} f \equiv 0$.
In view of Step 1, this implies that $f\equiv 0$. The proof is complete.

\section{Integral representation}

We are now in the position to prove Theorem \ref{thm:main}. We first consider the case $p=1$ with $\lambda=t$.

\begin{lemma}\label{lem:Prhopar}
Theorem \ref{thm:main} holds for $p=1$ with $\lambda=t$.
\end{lemma}
\begin{proof}
Assume that $f\in A_{\lambda}^1(\calU)$. Let $\gamma>\lambda$, by Lemma \ref{thm:DK} we have
\[
f(z)=c_{\gamma} \int\limits_{\calU} \frac {\bfrho(w)^{\gamma}}{\bfrho(z,w)^{n+1+\gamma}} f(w) dV(w)
\]
for all $z\in\calU$. Then a simple calculation shows that
\[
\bfL_n^N f(z) = (n+1+\gamma)_{N} (-i/2)^N c_{\gamma}  \int\limits_{\calU} \frac {\bfrho(w)^{\gamma}}{\bfrho(z,w)^{n+1+\gamma+N}} f(w) dV(w),
\]
where $(n+1+\gamma)_{N}$ is the Pochhammer symbol denoted by $(a)_{k}=a(a+1)\cdots(a+k-1)$.
Thus,
\begin{align*}
\mathcal{P}_{\lambda}(\bfrho^N \bfL_n^N f)(z)
&= c_{\lambda} \int\limits_{\calU} \frac{\bfrho(w)^{\lambda+N} \bfL_n^N f(w)}{\bfrho(z,w)^{n+1+\lambda}} dV(w)\\
&=  c_{\lambda} (n+1+\gamma)_{N} (-i/2)^N c_{\gamma} \\
& \quad\times  \int\limits_{\calU} \frac{\bfrho(w)^{\lambda+N}}{\bfrho(z,w)^{n+1+\lambda}}
\bigg(\int\limits_{\calU} \frac {\bfrho(u)^{\gamma}}{\bfrho(w,u)^{n+1+\gamma+N}} f(u) dV(u) \bigg) dV(w).
\end{align*}
By Fubini's theorem (the justification shall be stated later) and \eqref{eqn:keylem2}, the above double integral equals to
\begin{align*}
\int\limits_{\calU} & \bfrho(u)^{\gamma} f(u) \bigg(
\int\limits_{\calU} \frac{\bfrho(w)^{\lambda+N}}{\bfrho(z,w)^{n+1+\lambda} \bfrho(w,u)^{n+1+\gamma+N}} dV(w)
\bigg) dV(u)\\
&=  C_1(n,n+1+\lambda,n+1+\gamma+N,\lambda+N)
\int\limits_{\calU} \frac{\bfrho(u)^{\gamma} f(u)}{\bfrho(z,u)^{n+1+\gamma}} dV(u)\\
&= C_1(n,n+1+\lambda,n+1+\gamma+N,\lambda+N) c_{\gamma}^{-1} f(z)
\end{align*}
for every $z\in\calU$. Consequently,
\[
\mathcal{P}_{\lambda}(\bfrho^N \bfL_n^N f)(z)=\frac{(-i/2)^N\Gamma(1+\lambda+N)}{\Gamma(1+\lambda)} f(z).
\]
It remains to justify the use of Fubini's theorem. For any fixed $z\in\calU$, by \eqref{eqn:elemtryeq1} and \eqref{eqn:keylem} (note that $\gamma>\lambda$), it follows that
\begin{align*}
&\int\limits_{\calU} \bigg( \int\limits_{\calU} \frac{\bfrho(w)^{\lambda+N}}{|\bfrho(z,w)|^{n+1+\lambda} |\bfrho(w,u)|^{n+1+\gamma+N}} dV(w) \bigg) \bfrho(u)^{\gamma} |f(u)| dV(u)\\
&\leq (\bfrho(z)/2)^{-n-1-\lambda} \int\limits_{\calU} \bigg( \int\limits_{\calU} \frac{\bfrho(w)^{\lambda+N}}{|\bfrho(w,u)|^{n+1+\gamma+N}} dV(w) \bigg) \bfrho(u)^{\gamma} |f(u)| dV(u)\\
&=  (\bfrho(z)/2)^{-n-1-\lambda} C_2(n,n+1+\gamma+N,\lambda+N)
\int\limits_{\calU} |f(u)| \bfrho(u)^{\lambda}dV(u) <\infty
\end{align*}
due to the assumption $f\in A_{\lambda}^1(\calU)$.
The proof of the lemma is complete.
\end{proof}

\begin{proof}[Proof of Theorem \ref{thm:main}]
By Lemma \ref{lem:Prhopar}, it remains to deal with the cases $1<p<\infty$ with $p(\lambda+1)>t+1$ and $p=1$ with $\lambda>t$.
Put
\[
g:=f - \frac{(2i)^N\Gamma(1+\lambda)}{\Gamma(1+\lambda+N)}\mathcal{P}_{\lambda} (\bfrho^N \bfL_n^N f).
\]
Note from Proposition \ref{prop:bddnsofder} that $\bfL_n^N f\in A_{t+Np}^p(\calU)$. Thus, by Lemma \ref{thm3.1:DK}, $\mathcal{P}_{\lambda} (\bfrho^N \bfL_n^N f)$ belongs to $A_t^p(\calU)$ and so does $g$.
A simple calculation shows that
\begin{align*}
\bfL_n^N g(z) &= \bfL_n^N f(z) - c_{\lambda+N}
\int\limits_{\calU} \frac {\bfrho(w)^{\lambda+N}} {\bfrho(z,w)^{n+1+\lambda+N}} \bfL_n^N f (w) dV(w)\\
&= \bfL_n^N f(z) - P_{\lambda+N}(\bfL_n^N f )(z)=0
\end{align*}
by Lemma \ref{thm:DK}, since $p(\lambda+N+1)>t+Np+1$ for any case.
Therefore,  it follows from Theorem \ref{thm:main2} that $g \equiv 0$, completing the proof.
\end{proof}

\section{Proof of Theorem \ref{thm:main4}}
It suffices to prove that there exists a positive constant $C$ such that
\begin{equation}\label{eq:Derivative norm}
\|f\|_{p,t} \leq C  \|\bfrho^N \bfL_n^N f\|_{p,t}.
\end{equation}
Since if the above inequality is proved, merged with Proposition \ref{prop:bddnsofder}, the theorem follows immediately.
Let $\lambda$ be sufficient large such that $p(\lambda+1)>t+1$.
Observe from Theorem \ref{thm:main} that
\[
f=  \frac{(2i)^N\Gamma(1+\lambda)}{\Gamma(1+\lambda+N)}\mathcal{P}_{\lambda} (\bfrho^N \bfL_n^N f).
\]
This together with Lemma \ref{thm3.1:DK} implies \eqref{eq:Derivative norm}, as desired.


\section*{Appendix: Proof of the fact $H^p(\calU) \subset A^{\frac{(n+1)p}{n}}(\calU)$}
\setcounter{section}{1}
\setcounter{equation}{0}
\setcounter{theorem}{0}

\renewcommand{\thesection}{A}

This fact is much less known than its counterpart in the setting of the unit ball.
Since we have been unable to find this fact in the literature, we give the details for completeness.

We begin by recalling some basic facts about the Hardy space $H^p(\calU)$ from \cite{Ste65}.

\begin{lemma}\label{lem:Stein65a}
If $0<p<\infty$ then there is a constant $A>0$ such that
\begin{equation*}
\int\limits_{b\calU} \sup_{\epsilon >0} |f(u+\epsilon \bfi)|^p d\bfbeta(u)\leq A \| f\|^p_{H^p(\calU)}
\end{equation*}
holds for all $f\in H^p(\calU)$.
\end{lemma}

\begin{lemma}\label{lem:Stein65b}
If $0<p<\infty$ and $f\in H^p(\calU)$, then there exists an $f^b\in L^p(b\calU)$ such that
\begin{enumerate}
  \item [(i)] $f(u+\epsilon \bfi) \to f^b (u)$ for almost every $u\in b\calU$, as $\epsilon \to 0^{+}$;
  \item [(ii)] $f(\,\cdot\, + \epsilon \bfi) \to f^b$ in the $L^p(b\calU)$ norm, as $\epsilon \to 0^{+}$;
  \item [(iii)] $\left\|f^b\right\|_{L^p(b\calU)} = \|f\|_{H^p(\calU)}$.
\end{enumerate}
The function $f^b$ is called the boundary function of $f$.
\end{lemma}

The next lemma is also well known, see \cite[Proposition 2.6]{Kor65}.
\begin{lemma}\label{lem:Kor65}
If $1\leq p<\infty$ and $f\in H^p(\calU)$ then $f=P[f^b]$, where\[
P[f^b](z):=\frac{(n-1)!}{4\pi^n} \int\limits_{b\calU} \frac{\bfrho(z)^n}{|\bfrho(z,u)|^{2n}} f^b(u) d\bfbeta(u),\qquad z\in\calU,
\]
is the Poisson integral of $f^b$.
\end{lemma}


We shall need the following formula from \cite[Lemma 13]{Liu18}.
\begin{lemma}\label{lem:LiuCA18}
If $\theta>0$ then
\begin{equation*}\label{eqn:keylem1}
\int\limits_{b\calU} \frac {d\bfbeta(u)} {|\bfrho(z,u)|^{n+\theta}} ~=~
\frac {4\pi^{n} \Gamma(\theta)} {\Gamma^2\left(\frac {n+\theta}{2}\right)}\ \bfrho(z)^{-\theta}, \qquad z\in \calU.
\end{equation*}
\end{lemma}

\begin{lemma}\label{lem:pointestimate}
If $1\leq p<\infty$ and $f\in H^p(\calU)$ then
\[
|f(z)|\lesssim \bfrho(z)^{-n/p} \|f\|_{H^p(\calU)}
\]
for all $z\in\calU$.
\end{lemma}

\begin{proof}
We consider two separate cases.

\subsubsection*{Case 1: $1<p<\infty$}
For any fixed $z\in \calU$, by Lemma \ref{lem:Kor65} and H\"older's inequality, with $q:=p/(p-1)$,
\begin{align*}
|f(z)|&= \left|P[f^b](z)\right|
\lesssim \bigg\{\int\limits_{b\calU} \frac{\bfrho(z)^{nq}}{|\bfrho(z,u)|^{2nq}}d\bfbeta(u)\bigg\}^{1/q} \left\|f^b\right\|_{L^p(b\calU)}\\
&\lesssim \bfrho(z)^n \bfrho(z)^{-(2nq-n)/q}  \left\|f^b\right\|_{L^p(b\calU)} =\bfrho(z)^{-n/p} \|f\|_{H^p(\calU)},
\end{align*}
where in the second inequality we used Lemma \ref{lem:LiuCA18} and in the last equality we used Lemma \ref{lem:Stein65b} (iii).

\subsubsection*{Case 2: $p=1$}
By Lemma \ref{lem:Kor65}, together with the easy observation that $2|\bfrho(z,u)|\geq \bfrho(z)$ for all $u\in b\calU$, we have
\begin{align*}
|f(z)|=|P[f^b](z)| \lesssim~ & \bfrho(z)^{-n} \int\limits_{b\calU} \left|f^b(u)\right| d\bfbeta(u)\\
=~ & \bfrho(z)^{-n} \left\|f^b\right\|_{L^1(b\calU)} = \bfrho(z)^{-n} \|f\|_{H^1(\calU)},
\end{align*}
where, again, the last equality follows from Lemma \ref{lem:Stein65b} (iii).
\end{proof}

\begin{definition}
Let $N$ be a positive integer, and let $\mathfrak{A}_{N}$ be the family of functions $f$ in $H^{\infty}(\calU)$ satisfying
\begin{enumerate}
\item[(i)]
$f$ is continuous on $\overline{\calU}$;
\item[(ii)]
$ \sup \left\{|z_n+i|^N |f(z)|: z\in\overline{\calU} \right\} < +\infty$.
\end{enumerate}
\end{definition}

\begin{lemma}
If $1\leq p<\infty$, then the class $\mathfrak{A}_{N}$ is dense in $H^p(\calU)$.
\end{lemma}

\begin{proof}
Let $f\in H^p(\calU)$.
For $k=1,2,\ldots$, put $r_k:=1-1/k$ and
\[
G_k(\xi)= \left( \frac {r_k + \xi_n}{1 + r_k \xi_n} \right)^{N}.
\]
Note that 
$G_k$ are bounded by $1$, and with $N$ fixed they converge to $1$ uniformly on compact subsets
$\overline{\ball}\setminus \{-e_n\}$.
Set
\[
g_k(z):=G_k(r_k \Phi^{-1}(z)), \quad z\in \overline{\calU},
\]
where $\Phi^{-1}: \calU\to\ball$ is the inverse of Cayley transform given by
\[
\left(z^{\prime},z_{n}\right)\;\longmapsto\; \left(\frac {2iz^{\prime}}{i+z_{n}},
\frac {i-z_{n}}{i+z_{n}}\right).
\]
These functions satisfy
\begin{align}
& |g_k(z)|\leq 1 \text{ for all } z \in \overline{\calU};\\
& g_k(z) \to 1 \text{ as } k\to \infty, \text{ for every } z\in \overline{\calU}.
\end{align}


Now, we consider the sequence of functions
\[
f_k(z) := g_k(z) f\Big(z+\frac {1}{k}\, \bfi\Big), \qquad z\in \overline{\calU}, \; k=1,2,\ldots.
\]
In view of (A.2) and Lemma \ref{lem:Stein65b} (i), it is obvious that $f_k (u) \to f^b(u)$ for almost every $u\in b\calU$.
Also, it is immediate from (A.1) that
\[
\left|f_k(u)-f^b(u)\right| \leq  \sup_{\epsilon >0} |f(u+\epsilon \bfi)|+|f^b(u)|
\]
for all $u\in b\calU$. Thus, by the dominated convergence theorem, together with Lemmas A.1 and A.2, we find that
$\|f_k-f\|_{H^p(\calU)} = \|f_k - f^b\|_{L^p(b\calU)} \to 0$ as $k\to +\infty$.

It remains to show that $f_k$ belongs to $\mathfrak{A}_N$ for each $N$.
First, by Lemma \ref{lem:pointestimate}, we see that
\[
\left| f\Big(z+\frac {1}{k}\, \bfi\Big)\right| \lesssim \bfrho(z+\frac{1}{k}\bfi)^{-n/p}\|f\|_{H^p(\calU)}
\leq k^{n/p} \|f\|_{H^p(\calU)}.
\]
Then, a simple calculation shows that
\begin{align*}
|z_n+i|^N |f_k(z)|~=~&|z_n+i|^N \left|\frac{r_k [\Phi^{-1}(z)]_n+r_k}{1+r_k^2[\Phi^{-1}(z)]_n}\right|^{N}
\left| f\Big(z+\frac {1}{k}\, \bfi\Big)\right|\\
\lesssim~&  |z_n+i|^N \left|\frac{2i\, r_k }{(z_n+i)(1+r_k^2[\Phi^{-1}(z)]_n)}\right|^{N}
k^{n/p} \|f\|_{H^p(\calU)}\\
\leq~& \frac {2^N k^{n/p}}{(1-r_k^2)^{N}}  \|f\|_{H^p(\calU)} ~\leq~ 2^N k^{n/p+N} \|f\|_{H^p(\calU)},
\end{align*}
which completes the proof.
\end{proof}

%

\begin{lemma}\label{lem:relation}
Let $1\leq p<\infty$. Given a function $f$ on $\calU$, let
\begin{equation}
\tilde{f}(\xi) := c_{n,p} (1+\xi_n)^{-2n/p} f(\Phi(\xi)), \qquad \xi\in\ball,
\end{equation}
where $c_{n,p}:=\left(4\pi^n/(n-1)!\right)^{1/p}$ and $\Phi:\ball\to\calU$ is the Cayley transform given by
\[
(\xi^{\prime}, \xi_{n})\longmapsto\left( \frac {\xi^{\prime}}{1+\xi_{n}},
i\left(\frac {1-\xi_{n}}{1+\xi_{n}}\right) \right).
\]
Then $f\in H^p(\calU)$ if and only if $\tilde{f}\in H^p(\ball)$ and $\big\|\tilde{f}\big\|_{H^p(\ball)}=\|f\|_{H^p(\calU)}$.
\end{lemma}
\begin{proof}
We only prove the ``only if'' part, the proof of the ``if'' part is similar.

Recall that for $0<p<\infty$ the Hardy space $H^p(\ball)$ over the unit ball $\ball$ of $\mathbb{C}^n$ consists of holomorphic
functions $F$ on $\ball$ such that
\[
\|F\|_{H^p(\ball)}:=\sup_{0<r<1} \left\{ \int\limits_{\sphere} |F(r\zeta)|^p d\sigma(\zeta)\right\}^{1/p} <\infty,
\]
where $d\sigma$ be the normalized surface measure on $\sphere$. Also, any function $F$ in $H^p(\ball)$ has radial boundary limit
$F^{\ast}(\zeta):=\lim_{r\to 1^{-}} F(r\zeta)$ for almost every $\zeta\in \sphere$. Moreover, we have
\begin{equation}\label{eqn:hardy-bdy-norm}
\|F\|_{H^p(\ball)}=\|F^{\ast}\|_{L^p(\sphere)}.
\end{equation}

We first assume that $f\in\mathfrak{A}_{N}$ for some $N>2n/p$. Writing $z=\Phi(\xi)$, we have
\begin{align*}
\big|\tilde{f}(\xi)\big| = c_{n,p}  |z_n+i|^{2n/p} |f(z)|
\leq c_{n,p} \sup_{z\in \calU} \left\{ |z_n+i|^N |f(z)|\right\}
\end{align*}
for all $\xi\in\ball$,  so $\tilde{f}\in H^{\infty}(\ball)\subset H^p(\ball)$. In particular, $\tilde{f}$ has radial boundary limit $\big(\tilde{f}\big)^{\ast}(\zeta)$ for almost every $\zeta\in \sphere$. On the other hand, since $f\in C(\overline{\calU})$ and the Cayley transform $\Phi$ is a continuous
bijection from $\overline{\ball} \setminus \{-e_n\}$ to $\overline{\calU}$, the function $\tilde{f}$ can be extended to be a continuous function
on $\overline{\ball} \setminus \{-e_n\}$.
It follows that
\[
\big(\tilde{f}\big)^{\ast}(\zeta)=c_{n,p} (1+\zeta_n)^{-2n/p} f^b(\Phi(\zeta))
\]
for almost every $\zeta\in \sphere\setminus \{-e_n\}$. Thus, by \cite[p. 575, 7.2(b)]{Ste93} we obtain
%
%
\begin{align*}
\int\limits_{\sphere} \big|\big(\tilde{f}\big)^{\ast}(\zeta)\big|^p d\sigma(\zeta)
~=~& (c_{n,p})^p \int\limits_{\sphere} \big|f^b(\Phi(\zeta))\big|^p \frac {d\sigma(\zeta)}{|1+\zeta_n|^{2n}}
= \int\limits_{b\calU} |f^b (u)|^p d\bfbeta(u).
\end{align*}
This implies that $\big\|\tilde{f}\big\|_{H^p(\ball)}=\|f\|_{H^p(\calU)}$, in view of \eqref{eqn:hardy-bdy-norm} and Lemma \ref{lem:Stein65b} (iii).

We have shown that the mapping $\mathcal{T}: f\mapsto \tilde{f}$ is an isometry from $\mathfrak{A}_N$ into $H^p(\ball)$. By
density, there is a unique bounded extension of $\mathcal{T}$ from $H^p(\calU)$ to $H^p(\ball)$.
Let us denote this extension by $\widetilde{\mathcal{T}}$. Then $\widetilde{\mathcal{T}}$ is also an isometry from $H^p(\calU)$ to $H^p(\ball)$, i.e.,
\[
\|\widetilde{\mathcal{T}}f\|_{H^p(\ball)}=\|f\|_{H^p(\calU)}
\]
for all $f\in H^p(\calU)$. It remains to show that $\widetilde{\mathcal{T}}f=\tilde{f}$ for all $f\in H^p(\calU)$.

Let $f\in H^p(\calU)$. Then there is a sequence $\{f_k\}\in\mathfrak{A}_{N}$ converges to $f$ in $H^p(\calU)$.
In view of Lemma \ref{lem:pointestimate}, $f_k \to f$ pointwise as well, and hence $\widetilde{f_k}\to \tilde{f}$ pointwise.
For every $0<r<1$, by Fatou's lemma, we have
\begin{align*}
\int\limits_{\mathbb{S}} \big|\tilde{f}(r\zeta)\big|^p \sigma(\zeta)
&\leq \liminf_{k\to\infty} \int\limits_{\mathbb{S}} \big|\widetilde{f_k}(r\zeta)\big|^p \sigma(\zeta)\\
&\leq \liminf_{k\to\infty} \big\|\widetilde{f_k} \big \|^p_{H^p(\ball)}
= \liminf_{k\to\infty} \|f_k\|^p_{H^p(\calU)} =\|f\|^p_{H^p(\calU)},
\end{align*}
which implies that $\tilde{f}\in H^p(\ball)$ and $\big\|\tilde{f}\big\|^p_{H^p(\ball)}\leq \|f\|^p_{H^p(\calU)}$.
By the uniqueness of the extension $\widetilde{\mathcal{T}}$, we conclude that $\widetilde{\mathcal{T}}f=\tilde{f}$ for all $f\in H^p(\calU)$.
The proof is complete.
%
\end{proof}

\begin{theorem}
If $1\leq p < \infty$ then $H^p(\calU) \subset A^{\frac{(n+1)p}{n}}(\calU)$.
\end{theorem}

\begin{remark}
It is quite natural to ask whether this result could be extended to the case $0<p<1$.  However, we have been unable to show
this.
\end{remark}


\begin{proof}
Let $f\in H^p(\calU)$ and $\tilde{f}$ be as in Lemma \ref{lem:relation}. Then $\tilde{f} \in H^p(\ball)$.
Recalling that $H^p(\ball)\subset A^{\frac{(n+1)p}{n}}(\ball)$ (see for instance \cite[Theorem 4.48]{Zhu05}),
we find that $\tilde{f}\in A^{\frac{(n+1)p}{n}}(\ball)$, and hence
\begin{equation}\label{eqn:Lpintegrablility}
\int\limits_{\ball} |f(\Phi(\xi))|^{\frac{(n+1)p}{n}} \frac {dV(\xi)}{|1+\xi_n|^{2(n+1)}}
~=~ (c_{n,p})^{-\frac{(n+1)p}{n}}  \big\|\tilde{f}\big\|_{L^{\frac{(n+1)p}{n}}(\ball)}^{\frac{(n+1)p}{n}}  ~<~ +\infty.
\end{equation}
On the other hand, after the change of variables $z=\Phi(\xi)$ in the integral, we find that
\begin{align*}
\int\limits_{\ball} |f(\Phi(\xi))|^{\frac{(n+1)p}{n}} \frac {dV(\xi)}{|1+\xi_n|^{2(n+1)}}  ~=~ \frac {1}{4} \int\limits_{\calU} |f(z)|^{\frac{(n+1)p}{n}} dV(z),
\end{align*}
which, together with \eqref{eqn:Lpintegrablility}, implies that $f\in A^{\frac{(n+1)p}{n}}(\calU)$ as desired.
\end{proof}


\end{document}